\documentclass[12pt, twoside, eqno]{article}
\usepackage{}

\usepackage{amsmath,amsthm}
\usepackage{amssymb,latexsym}
\usepackage{enumerate}
\usepackage{amssymb}
\usepackage{mathrsfs}
\usepackage{amsfonts}

\newtheorem{thm}{Theorem}[section]
\newtheorem{lemma}{Lemma}[section]
\newtheorem{dfn}{Definition}[section]
\newtheorem{cor}{Corollary}[section]

\def\eps{{\varepsilon}}
\def\Bad{{\rm Bad}}

\allowdisplaybreaks[3]

\usepackage[pagewise]{lineno}

\begin{document}

\baselineskip=17pt


\title{On homogeneous and inhomogeneous Diophantine approximation over the f\/ields of formal
power series}

\author{YANN BUGEAUD, ZHENLIANG ZHANG}

\date{}

\maketitle


\renewcommand{\thefootnote}{}



\renewcommand{\thefootnote}{\arabic{footnote}}
\setcounter{footnote}{0}


\begin{abstract}

We prove over fields of power series the analogues of several Diophantine
approximation results obtained over the field of real numbers.
In particular we establish the power series analogue
of Kronecker's theorem for matrices, together with a
quantitative form of it, which can also be seen as a transference inequality between
uniform approximation and inhomogeneous approximation.
Special attention is devoted to the one dimensional case. Namely,
we give a necessary and sufficient condition on an irrational power series $\alpha$
which ensures that, for some positive $\eps$, the set
$$
\liminf_{Q \in \mathbb{F}_q[z], \,\, \deg Q \to \infty}
 \| Q \| \cdot \min_{{y}\in \mathbb{F}_q[z]}\|Q \alpha - \theta - y\|
\geq \eps
$$
has full Hausdorff dimension.

\end{abstract}

\section{Introduction}

Let $q$ be a power of a prime number $p$ and $\mathbb{F}_q$ the finite field of order $q$.
Recall that $\mathbb{F}_q[z]$ and $\mathbb{F}_q(z)$ denote the ring of polynomials and the field
of rational functions over $\mathbb{F}_q$, respectively.
Let $\mathbb{F}_q((z^{-1}))$ denote the field of formal power series
$x=\sum_{i=-n}^\infty a_{i} z^{-i}$ over the field $\mathbb{F}_q$.
We equip $\mathbb{F}_q((z^{-1}))$ with the norm $\|x\|=q^n$,
where $a_{-n} \neq 0$ is the first non-zero coefficient in the expansion of the
non-zero power series $x$. This integer $n$ is called the {{degree}} of $x$ and denoted by $\deg x$.

The sets
$\mathbb{F}_q[z]$, $\mathbb{F}_q(z)$, and $\mathbb{F}_q((z^{-1}))$
play the roles of $\mathbb{Z}, \mathbb{Q},$ and $\mathbb{R}$, respectively.
A power series $x$ in $\mathbb{F}_q((z^{-1}))$ but not in $\mathbb{F}_q(z)$ is called irrational.
We denote by $[x]$ and $\{x\}$ the ``integral part'' and the ``fractional part'' of the power series
$x=\sum_{i=-n}^\infty a_{i}z^{-i}$ in $\mathbb{F}_q((z^{-1}))$, defined as
$$
[x]=\sum_{i=-n}^0 a_{i}z^{-i},\;\;\;\{x\}=\sum_{i=1}^{\infty} a_{i}z^{-i}.
$$
In particular, $[x]$ is a polynomial in $z$.

Let $\mathbb{I}=\left\{x \in \mathbb{F}_q((z^{-1}))\colon\|x\|<1\right\}$ be the open unit ball.
A natural measure on $\mathbb{I}$ is the normalized
Haar measure on $\Pi_{n=1}^{\infty}\mathbb{F}_q$,
which we denote by  $\mu$. Observe that $\mu (\mathbb{I}) = 1$.
If $B(x, q^{-r})$ is the open ball of center $x$ in $\mathbb{I}$ and radius $q^{-r}$,
namely
$$
B(x,r)=\{y\in \mathbb{I}\colon \|y-x\|< q^{-r}\},
$$
then $\mu (B(x,q^{-r}))=q^{-r}$.
Since the norm $\| \cdot \|$ is non-Archimedean, any two balls $C_1$ and $C_2$
satisfy either $C_1\cap C_2=\emptyset$, $C_1\subset C_2$, or $C_2\subset C_1$.
This is sometimes referred to as \emph{the ball intersection property}.
Moreover, the distance between any two disjoint balls is not less
than the maximal radius of the two balls.

For any (column) vector $\underline{\theta}$ in $\mathbb{F}_q((z^{-1}))^n$,
we denote by $\|\underline{\theta}\|$ the maximum of the norm of its coordinates
and by
$$
|\langle \underline{\theta}\rangle|
=\min_{\underline{y}\in \mathbb{F}_q[z]^n}\|\underline{\theta}-\underline{y}\|
$$
the maximum of the distances of its coordinates to their integral parts.   

There are numerous results on Diophantine approximation in the fields of formal power series,
see \cite{Lasjaunias} and Chapter 9 of \cite{Bugeaud1} for references;
more recent works include \cite{Bank,GaGh17,GaGh19,Kirstensen3,zhang,zheng}.
However, only few results are known on the relation between
homogenous and inhomogeneous Diophantine approximation.
Our first result is the analogue of Kronecker's theorem over fields of formal power series.
As far as we are aware, it has not yet been proved in such a generality (see, however,
\cite{Car52,Mahler} for the case of column matrices).
The transposed matrix of a matrix $A$ is denoted by $A^T$.

\begin{thm}
\label{kronecker}
Let $m, n$ be positive integers.
Let $A$ be in $\mathcal{M}_{n,m}(\mathbb{F}_q((z^{-1})))$ and
$\underline{\theta}$ in $\mathbb{F}_q((z^{-1}))^n$.
Then the following two statements are equivalent:
\begin{itemize}
  \item [(1)] For every $\eps>0$, there exists a
  polynomial vector $\underline{x}$ in $\mathbb{F}_q[z]^m$ such that
  $$
  |\langle A\underline{x}-\underline{\theta}\rangle|\leq \eps.
  $$
  \item [(2)] If $\underline{u}=(u_1,\cdots,u_n)^T$ is any polynomial vector
  such that $A^T \underline{u}$ is in $ \mathbb{F}_q[z]^m$, then
  $$ u_1\theta_1+\cdots+u_n\theta_n\in \mathbb{F}_q[z].$$
\end{itemize}
\end{thm}

As in \cite{Bugeaud3}, which deals with the real case, our aim is to give a
quantitative version of Theorem \ref{kronecker}.
Following \cite{Bugeaud3}, we introduce several exponents of homogeneous
and inhomogeneous Diophantine approximation.
Let $n$ and $m$ be positive integers
and $A$ a matrix in $\mathcal{M}_{n,m}(\mathbb{F}_q((z^{-1})))$.
Let $\underline{\theta}$ be in $\mathbb{F}_q((z^{-1}))^n$.
We denote by $\omega(A,\underline{\theta})$ the supremum
of the real numbers $\omega$ for which,
for arbitrarily large real numbers $H$,
the inequalities
\begin{equation}
\label{exponent}
|\langle A\underline{x}-\underline{\theta}\rangle|\leq H^{-\omega}\;\;\;
\text{and }\;\;\;\|\underline{x}\|\leq H
\end{equation}
have a solution $\underline{x}$ in $\mathbb{F}_q[z]^m$.
Let $\widehat{\omega}(A,\underline{\theta})$ be the
supremum of the real numbers $\omega$ for which,
for all sufficiently large positive real numbers $H$,
the inequalities (\ref{exponent}) have a solution $\underline{x}$ in $\mathbb{F}_q[z]^m$.
It is obvious that
$$
\omega(A,\underline{\theta})\geq \widehat{\omega}(A,\underline{\theta})\geq0.
$$
We define furthermore two homogeneous exponent $\omega(A)$ and $\widehat{\omega}(A)$
as in (\ref{exponent})
when $\underline{\theta}$ is the zero vector,
requiring moreover that the polynomial solution $\underline{x}$ should be non-zero.

Our second result is the power series analogue of the main result of \cite{Bugeaud3}.   
Throughout this paper, the quantity $1/+\infty$ is understood to be $0$.

\begin{thm}
\label{mainthm}
Let $m, n$ be positive integers.
Let $A$ be in $\mathcal{M}_{n,m}(\mathbb{F}_q((z^{-1})))$ and
$\underline{\theta}$ in $\mathbb{F}_q((z^{-1}))^n$.
Then, we have the lower bounds
\begin{equation}
\label{theorem inequalities}
\omega(A,\underline{\theta})\geq  \frac{1}{\widehat{\omega}(A^T)} \;\;\;  \text{and}\;\;\;
\widehat{\omega}(A,\underline{\theta})\geq  \frac{1}{\omega(A^T)},
\end{equation}
with equalities in (\ref{theorem inequalities}) for almost all $\underline{\theta}$
with respect to the Haar measure on $\mathbb{F}_q((z^{-1}))^n$.If $\underline{\theta}$ is not in
$A\mathbb{F}_q[z]^m+\mathbb{F}_q[z]^n$, then we also have the upper bound
$$
\widehat{\omega}(A,\underline{\theta}) \le \omega (A).
$$
\end{thm}

If the subgroup $G_A=A^T\mathbb{F}_q[z]^n+\mathbb{F}_q[z]^m$ of $\mathbb{F}_q((z^{-1}))^m$
has rank $\text{rk}_{\mathbb{F}_q[z]} (G_A)$ smaller than $m+n$,
then there exists $\underline{x}$ in $\mathbb{F}_q[z]^n$ with arbitrarily large norm
such that $|\langle A^T\underline{x}\rangle|=0$ and we have
$$
\widehat{\omega}(A^T) =  {\omega}(A^T) = + \infty .
$$
Throughout the paper, we avoid this degenerate case and
consider only matrices $A$ for which $\text{rk}_{\mathbb{F}_q[z]} (G_A)=m+n$.

Kim and Nakada \cite{Kim1} proved that, for any $\alpha$ in $\mathbb{I}$, we have
$$
\liminf_{n\to \infty}\left(q^n \min_{\deg Q=n}\|\{Q\alpha\}-\beta\|\right)=0,
$$
for almost all $\beta$ in $\mathbb{I}$.
In a subsequent paper \cite{Kim2}, the authors
complemented this result in showing that,
for any irrational power series $\alpha$ in $\mathbb{I}$, the set
$$
\left\{\beta\in \mathbb{I} \colon \liminf_{n\to \infty}
\left(q^n \min_{\deg Q=n}\|\{Q\alpha\}-\beta\|\right)>0\right\}
$$
has full Hausdorff dimension.
Our next result generalizes this statement to matrices of arbitrary dimension.
Before stating it, we introduce the following notations.

Let $m, n$ be positive integers
and $A$ in $\mathcal{M}_{n,m}(\mathbb{F}_q((z^{-1})))$.
For $\eps>0$, we define the set
$$
\Bad^{\eps} (A)
:= \{\underline{\theta}\in \mathbb{I}^n :  \liminf_{\underline{x}\in \mathbb{F}_q[z]^m,
\, \, \|\underline{x}\| \to \infty}
 \|\underline{x}\|^{m/n} \cdot |\langle A \underline{x} - \underline{\theta} \rangle| \ge \eps \}
$$
and we put
$$
\Bad(A) := \underset{\eps>0}{\bigcup} \, \Bad^\eps (A)
= \{\underline{\theta}\in \mathbb{I}^n :  \liminf_{\underline{x}\in \mathbb{F}_q[z]^m,
\, \, \|\underline{x}\| \to \infty}
\|\underline{x}\|^{m/n} \cdot |\langle A \underline{x} - \underline{\theta} \rangle|  > 0  \}.
$$
When $n=m=1$ and $A = (\alpha)$ we simply write $\Bad^{\eps} (\alpha)$
and $\Bad (\alpha)$ instead of
$\Bad^{\eps} (A)$ and $\Bad (A)$.

\begin{thm}
\label{prop}
Let $m, n$ be positive integers.
For any matrix $A$ in $\mathcal{M}_{n,m}(\mathbb{F}_q((z^{-1})))$,
the set $\Bad (A)$ has full Hausdorff dimension. More precisely,
there exists a continuous function $f : \mathbb{R}_+ \to \mathbb{R}_+$ such that $f(0) = 0$ and
the Hausdorff dimension of the set $\Bad^\eps (A)$ is at least $n - f(\eps)$,
for every positive $\eps\leq q^{-\frac{m}{n}-6}$.
\end{thm}

If the sequence of the norms of the
best approximation vectors associated to $A$ (see Definition \ref{bestapp})
increases sufficiently rapidly, then
the above results can be strengthened as follows.
Similar results in the real case have been established in \cite{Bugeaud2}.

\begin{thm}
\label{higherdim}
Let $m, n$ be positive integers.
Let $A$ be in $\mathcal{M}_{n,m}(\mathbb{F}_q((z^{-1})))$ and $(\underline{y}_k)_{k\geq1}$
the sequence of best approximation vectors associated to $A$.
If $\|\underline{y}_k\|^{\frac{1}{k}}$ tends to infinity with $k$,
then there exists a positive real number $\varepsilon$ such that the set
$\Bad^\eps (A)$
has full Hausdorff dimension.
More precisely,
 $\varepsilon$ can be taken to be any positive real number less than $q^{-4-\frac{m}{n}}$.
Moreover,
if $m=n=1$, $A = (\alpha)$,
and the degree of the partial quotients in the continued fraction expansion of $\alpha$
in $\mathbb{F}_q((z^{-1}))$ tends to infinity, then
the set $\Bad^\eps (\alpha)$ has full Hausdorff dimension for every $\eps\le q^{-2}$.
\end{thm}

Except for $(m, n) = (1, 1)$ (see the next section), we do not know whether the condition
``$\|\underline{y}_k\|^{\frac{1}{k}}$ tends to infinity with $k$'' is necessary to ensure that
$\Bad^\eps (A)$ has full Hausdorff dimension for some positive $\eps$.

The present paper is organized as follows.
In Section 2,
we give additional results in the one dimensional case,
including necessary and sufficient conditions to ensure
that the set ${\Bad}^{\,\varepsilon}(\alpha)$ has full Hausdorff dimension.
In Section 3, we present some auxiliary results.
A transference lemma is established in Section 4,
where we also give the proof of Theorem \ref{kronecker}.
The proofs of Theorem \ref{mainthm}, Theorem \ref{prop}, and Theorem \ref{higherdim} are given in
Section 5, Section 6, and Section 7, respectively.
We use similar arguments as in the real case.
In Section 8, we prove Theorem \ref{onedimthm}.
The proofs of Theorem \ref{main theorem2} and Theorem \ref{main theorem3}
are postponed to the last two sections.


\section{One dimensional case}

In the one dimensional case, Theorem \ref{higherdim} can be complemented as follows.

\begin{thm}
\label{main theorem2}
Let $\alpha$ be an irrational power series in $\mathbb{F}_q((z^{-1}))$ and
$Q_k$ the denominator of its $k$-th convergent for $k\ge 1$.
Then, there exists $\varepsilon>0$ such that the set ${\Bad}^{\,\varepsilon}(\alpha)$
has full Hausdorff dimension
if and only if  $\lim_{k\to\infty}\|Q_k\|^\frac{1}{k}=\infty$.
\end{thm}

In addition, 
we give a third condition equivalent to those occurring in Theorem \ref{main theorem2}. 
For an irrational power series $\alpha$ in $\mathbb{I}$ and a positive real number $c$,
let $\Delta_{N, c} (\alpha)$ denote the number of integers $l$ in $ \{1,\cdots,N\}$
for which the inequality $ \|\{Q\alpha\}\|\le c2^{-l} $
has a solution $ Q$ in $\mathbb{F}_q[z]$ with $0<\|Q\|\le 2^l$.
Then, the power series $\alpha$ is called singular on average if, for every $c>0$,
we have $\lim_{N\to\infty}\frac{1}{N}\Delta_{N,c} (\alpha)=1$.
As far as we are aware, this notion has been introduced in \cite{KKLM}.

\begin{thm}
\label{main theorem3}
Let $\alpha$ be an irrational power series.
There exists $\varepsilon>0$ such that the set ${\Bad}^{\,\varepsilon}(\alpha)$
has full Hausdorff dimension
if and only if $\alpha$ is singular on average.
\end{thm}

Theorems \ref{main theorem2} and \ref{main theorem3} are the power series analogues of
Theorem 1.1 of \cite{Bugeaud2}. 
In the proof of Theorem \ref{main theorem2}, our method
is different: we replace the use of the three distance theorem in \cite{Bugeaud2} by
that of Ostrowski expansions, see Theorem \ref{Ostro} and its proof.
Theorem \ref{main theorem3} is proved in a similar way as in the real case. 

Our last result gives additional information about the relation between the exponents of homogeneous
and inhomogeneous Diophantine approximation in dimension one.
Its first statement has already been established in Theorem \ref{mainthm}.

\begin{thm}
\label{onedimthm}
Let $\xi$ in $ \mathbb{F}_q((z^{-1}))$ be an irrational power series.
For any element $\theta$ in $\mathbb{F}_q((z^{-1}))$ not in $\mathbb{F}_q[z]+\xi\mathbb{F}_q[z]$,
we have
$$
\frac{1}{\omega((\xi))}\le \widehat{\omega}((\xi),\theta)\le\omega((\xi)).
$$
Let $\omega$ denote $+\infty$ or a real number greater than or equal 1,
then there exists a $\xi$ in $\mathbb{F}_q((z^{-1}))$ for which $\omega((\xi))=\omega$
and the set of values taken by the function $\widehat{\omega}((\xi),\cdot)$ is equal to the interval
$\left[\frac{\displaystyle1}{\displaystyle\omega},\omega\right]$.
\end{thm}

Theorem \ref{onedimthm} is the power series analogue of Proposition 8 of \cite{Bugeaud3}
and its proof uses similar arguments.


\section{Preliminaries}

In this section, we brief\/ly recall some notations and classical results which
will be used later in the proofs
of our theorems.

In the setting of formal power series,
every irrational element $\alpha$ in $\mathbb{I}$
has a unique infinite continued fraction expansion over the field $\mathbb{F}_q((z^{-1}))$,
which is induced by the map
$$
T\alpha=\frac{1}{\alpha}-\left[\frac{1}{\alpha}\right].
$$
The reader is referred to Artin \cite{Artin} or Berth\'{e} and Nakada \cite{Nakada} for more details.
For every irrational power series $\alpha$ in $\mathbb{I}$,
we denote by $\alpha=[0; A_1,A_2,\cdots] $ its continued fraction expansion,
where $A_k=A_k(\alpha)\colon=[\frac{1}{T^{k-1}\alpha}]$
is called the $k$-th partial quotient of $\alpha$.
For each $k\ge1$,
${P_k (\alpha)} / {Q_k (\alpha)}=[0, A_1,A_2,\cdots,A_k]$ is the $k$-th convergent of $\alpha$. 
This defines $P_k (\alpha)$ and $Q_k (\alpha)$ up to a common multiplicative factor. 
To define numerator and denominator of the $k$-th convergent of $\alpha$, we set 
$P_{-1}(\alpha)=Q_0 (\alpha)= 1$ and $Q_{-1}(\alpha)=P_0 (\alpha)= 0$,
and, for any $k \ge 0$, 
\begin{equation*}
\begin{split}
&P_{k+1}(\alpha)=A_{k+1}(\alpha)P_k(\alpha)+P_{k-1}(\alpha),\\
&Q_{k+1}(\alpha)=A_{k+1}(\alpha)Q_k(\alpha)+Q_{k-1}(\alpha).
\end{split}
\end{equation*}

The following elementary properties of continued fraction expansions of formal power
series are well-known (see Fuchs \cite{Fuchs} for details).

\begin{lemma}
[\cite{Fuchs}] \label{lemma fuchs}
Under the above notation, we have for $k \ge 1$:

\begin{enumerate}
\item[(1)] $\left(P_k(\alpha),Q_k(\alpha)\right)=1$,
\item[(2)] $1=\|Q_0(\alpha)\|<\|Q_1(\alpha)\|<\|Q_2(\alpha)\|<\cdots$,
\item[(3)] $\|Q_k(\alpha)\|=\Pi_{i=1}^k\| A_{i}(\alpha)\|$,
\item[(4)] $P_{k-1}(\alpha)Q_k(\alpha)-P_k(\alpha)Q_{k-1}(\alpha)=(-1)^k$.
\end{enumerate}
\end{lemma}

We also need a version of Dirichlet's theorem in the fields of formal power series.
The next statement follows from Theorem 2.1 of \cite{GaGh17}.

\begin{thm}
\label{Minkowski2}
Let $m, n$ be positive integers.
Let $A$ be in $\mathcal{M}_{n,m}(\mathbb{F}_q((z^{-1})))$.
Then, for any positive integer $c$,
there is a non-zero polynomial vector $\underline{u}$ such that
$$
|\langle A \underline{u} \rangle| < q^{-c\frac{m}{n}}
\;\;\;\text{and}\;\;\;
1\le \| \underline{u}\|\le q^c.
$$

\end{thm}

In dimension greater than one,
we deal with sequences of vectors having similar properties as the sequence of convergents
in dimension one.
For this purpose, for a matrix $A=(\alpha_{i,j})_{1 \le i \le n, 1 \le j \le m}$,
we denote by
$$
M_j(\underline{y})=\sum_{i=1}^n\alpha_{ij}y_i,\;\;\;\underline{y}=(y_1, \cdots, y_n)^T, \;\;\;(1\le j\le m)
$$
the linear forms determined by its columns.
Then, for $\underline{y}$ in $\mathbb{F}_q((z^{-1}))^n$, we set
$$
M(\underline{y})=\max_{1\le j\le m}|\langle M_j(\underline{y})\rangle|=|\langle A^T\underline{y}\rangle|.
$$

\begin{dfn}
\label{bestapp}
For a sequence of polynomial vectors $(\underline{y}_i)_{i\ge 1}$,
write
$$
\|\underline{y_i}\|=Y_i,\;\;\; M_i=M(\underline{y}_i).
$$
If the sequence satisfies
$$
1=Y_1<Y_2<\cdots ,  \;\;\;\ M_1>M_2>\cdots
$$
and
$M(\underline{y})\ge M_i$ for all non-zero polynomial vectors $\underline{y}$
of norm $\|\underline{y}\|<Y_{i+1}$,
then it is called a sequence of best approximations related to the matrix $A^T$ (or to the linear forms
$M_1,M_2,\cdots,M_m)$.
\end{dfn}

Now we construct inductively a sequence of best approximations related to the matrix $A^T$.

Let $Y_1=\|\underline{y}_1\|=1$,
and $M(\underline{y})\ge M(\underline{y_1})=M_1$
for any polynomial vector $\underline{y}$ in $ \mathbb{F}_q[z]^n$
with $\|\underline{y}\|=1$.

Suppose that $\underline{y}_1, \cdots,\underline{y}_i$ have already been constructed in such a way
that $M(\underline{y})\ge M_i$ for all non-zero polynomial vectors $\underline{y}$
with $\|\underline{y}\|\le Y_i$.
Let $Y$ be the smallest integer power of $q$ greater than $Y_i$ and for which
there exists a polynomial vector $\underline{z}$ with $\|\underline{z}\|=Y$ and $M(\underline{z})<M_i$.
Since $M_i$ is positive, the integer $Y$ does exist by Theorem \ref{Minkowski2}.
Among those points $\underline{z}$,
we select an element $\underline{y}$ for which $M(\underline{z})$ is minimal.
Then we set
$$
\underline{y}_{i+1}=\underline{y},\;Y_{i+1}=Y,\;\text{and}\;M_{i+1}=M(\underline{y}).
$$
The sequence $(\underline{y}_i)_{i \ge 1}$ constructed in this way enjoys the desired properties.

The following two lemmas collect some properties of the sequence of best approximations.
\begin{lemma}
\label{property of best appro}
Let $(\underline{y}_i)_{i\geq1}$ be the sequence of best approximations related to the linear forms
$M_1, \cdots,M_m$. Then we have
\begin{itemize}
  \item[(i)] $Y_i\ge q^i$, for $i \ge 1$.
  \item[(ii)] $M_i <  q^{\frac{n}{m}} \, Y_{i+1}^{-\frac{n}{m}}$, for $i \ge 1$.
  \item[(iii)] For $\omega<\widehat{\omega}(A^T)$, $M_i\leq  Y_{i+1}^{-\omega}$
  holds for any sufficiently large $i$.
  \item[(iv)] For $\omega<\omega(A^T)$, $M_i\leq Y_i^{-\omega}$ holds for infinitely many $i$.
  \end{itemize}
\end{lemma}

Remark.
In the special case $m=1$, $(ii)$ can be replaced by the large inequality
$M_i \le q^{n - 1} Y_{i+1}^{-n}$.

\begin{proof}
$(i)$ is immediate since $Y_{i+1}\ge q Y_i$.

$(ii)$. It follows from Theorem \ref{Minkowski2} that
the system of inequalities
$$
M(\underline{y}) < q^{-c\frac{n}{m}}\;\;\;\text{and}\;\;\;\|\underline{y}\|\le q^c
$$
has a non-zero polynomial $\underline{y}$ for $q^c=q^{-1}Y_{i+1}$. This implies that
$M_i <  (q^{-1}Y_{i+1})^{-\frac{n}{m}}$, as asserted.

$(iii)$. Let $\omega$ with $0 < \omega < \widehat{\omega}(A^T)$.
Then, the system of inequalities
$$
M(\underline{y})\le H^{-\omega}\;\;\;\text{and}\;\;\; \|\underline{y}\|\le H
$$
has a non-zero solution for any sufficiently large real number $H$. In particular, for
every sufficiently large integer $i$, the system of inequalities
$$
M(\underline{y})\le Y_{i+1}^{-\omega}\;\;\;\text{and}\;\;\; \|\underline{y}\| < Y_{i+1}
$$
has a non-zero solution $\underline{z}_i$, satisfying
$$
M_i\le M(\underline{z}_i)\le Y_{i+1}^{-\omega}.
$$

$(iv)$. For  $\omega<\omega(A^T)$,
there are infinitely many polynomial vectors $\underline{h}$ in $ \mathbb{F}_q((z^{-1}))^n$
such that $M(\underline{h})\le \|\underline{h}\|^{-\omega}$.
For every such $\underline{h}$ in $ \mathbb{F}_q((z^{-1}))^n$,
there exists an index $i$ such that $Y_i\le \|\underline{h}\|<Y_{i+1}$.
Then, $M_i\le M(\underline{h})\le \|\underline{h}\|^{-\omega}\le Y_i^{-\omega}$.
\end{proof}

\begin{lemma}
\label{bestlemma}
Let $(\underline{y}_i)_{i\geq1}$ be the sequence of best approximations related to the linear forms
$M_1,\cdots,M_m$. Then,
for almost all $\underline{\theta}=(\theta_1, \cdots,\theta_n)^T$ in $\mathbb{F}_q((z^{-1}))^n$,
we have
$$
|\langle \underline{y}_i  \underline{\theta}  \rangle|\ge Y_i^{-\delta},
$$
for any $\delta>0$
and any index $i$ which is sufficiently large in terms of $\delta$ and $\underline{\theta}$.
\end{lemma}

\begin{proof}
For any $\delta>0$ and any $i \ge 1$,
consider the set
$$
B(\underline{y_i})=\{\underline{\theta}=(\theta_1, \cdots,\theta_n)^T\colon
|\langle  \underline{y}_i  \underline{\theta}  \rangle|< Y_i^{-\delta}\}.
$$
It follows from equality (2.3) in \cite{Kirstensen3} that
the Haar measure of $B(\underline{y_i})$ is bounded from above by $Y_i^{-\delta}$
times some absolute, positive constant.
Combined with the fact that $Y_i\ge q^i$ for $i \ge 1$,
which ensures that the series $\sum_{i\ge 1}Y_i^{-\delta}$ converges,
we deduce from the Borel--Cantelli Lemma that
the set of $\underline{\theta}$ which belong to infinitely many sets $B(\underline{y_i})$
has Haar measure zero. This implies the lemma.
\end{proof}

Let $\alpha$ be in $\mathbb{I}$. Denote by $[0 ; A_1,A_2,\cdots] $ its continued fraction expansion
and by $\frac{P_k}{Q_k}$ its $k$-th convergent, for $k \ge 0$. Set
$$
D_k=Q_k \alpha-P_k, \quad \hbox{for $k \ge 1$.}
$$

\begin{lemma}
[\cite{Fuchs}]
\label{lemma1}
Under the above notation, we have
\begin{enumerate}
\item[(1)] $D_{k+1}=A_{k+1}D_k+D_{k-1},$
\item[(2)] $\|D_k\|=\|Q_k\alpha-P_k\|=\|\{Q_k\alpha\}\|=\frac{1}{\|Q_{k+1}\|}.$
\end{enumerate}
\end{lemma}

In addition to continued fractions,
we also make use of the Ostrowski expansion of the elements of $\mathbb{I}$
with respect to an irrational
power series $\alpha$.

.

\begin{lemma}
[\cite{Kim1}]
\label{lemma2}
Under the above notation,
for every positive integer $k$ and every $Q$ in $\mathbb{F}_q[z]$ with $\deg Q< \deg Q_{k+1}$,
there is a unique decomposition
$$
Q=B_1Q_0+B_2Q_1+\cdots+B_{k+1}Q_k,
$$
where $B_i$ is in $ \mathbb{F}_q[z]$ and $\deg B_i<\deg A_i$ for $1\le i\le k+1$.
\end{lemma}

\begin{lemma}
[\cite{Kim2}]
\label{lemma3}
Under the above notation, for every $\beta$ in $\mathbb{I}$,
there is a representation of $\beta$ under the form
\begin{equation}
\label{Ostrowski expansion}
\beta=\sum_{k=0}^{\infty}\sigma_{k+1}(\beta)D_k=\sigma_1(\beta)D_0+\sigma_2(\beta)D_1+\cdots,
\end{equation}
where $\sigma_i(\beta)$ is in $ \mathbb{F}_q[z]$
and $\deg \sigma_i(\beta)<\deg A_i(\alpha)$ for $ i\ge 1$.
The representation (\ref{Ostrowski expansion}) is called the
Ostrowski expansion of $\beta$ with respect to $\alpha$
or an $\alpha$-expansion for $\beta$.
\end{lemma}

For simplicity,
we write
$$
\beta=[\sigma_1(\beta),\sigma_2(\beta),\cdots,\sigma_n(\beta),\cdots]_{\alpha}
$$
and call the sequence $(\sigma_n(\beta))_{n\ge 1}$ the sequence of digits of $\beta$.
To facilitate the exposition,
we make use of a kind of symbolic space defined as follows.

For any $n\ge 1$, set
$$
\mathbb{L}_n(\alpha)=\{(\sigma_1, \cdots,\sigma_n)\colon \sigma_i\in \mathbb{F}_q[z]\;\;
\text{and}\;\;\deg \sigma_i<\deg A_i(\alpha)\;\;\text{for}\;\;1\le i\le n\}
$$
and
$$
\mathbb{L}(\alpha)=\bigcup_{n=1}^{\infty}\mathbb{L}_n(\alpha).
$$
Then, for any $(\sigma_1, \cdots,\sigma_n)$ in $\mathbb{L}_n(\alpha)$,
there exists an element $\beta$ in $\mathbb{I}$ whose sequence
of digits begins with $(\sigma_1, \cdots,\sigma_n)$.

For an $n$-tuple $\sigma=(\sigma_1, \cdots,\sigma_n)$ in $\mathbb{L}_n(\alpha)$,
we call
$$
I_n(\sigma_1, \cdots,\sigma_n)=\{\beta\in \mathbb{I}\colon \sigma_k(\beta)
=\sigma_k \;\;\text{for}\;\;1\le k\le n\}
$$
a cylinder of order $n$; this is
the set of formal power series in $\mathbb{I}$ which have an $\alpha$-expansion beginning with
$\sigma_1, \cdots,\sigma_n$.

For the size of the cylinder, we have the following lemma.

\begin{lemma}
[\cite{Kim2}]
\label{length of cylinder}
For any $\sigma=(\sigma_1,\cdots,\sigma_n)$ in $\mathbb{L}_n(\alpha)$,
the $n$-th cylinder $I_n(\sigma_1,\cdots,\sigma_n)$ is a closed disc
centered at $\sum_{k=0}^{n-1}\sigma_{k+1}D_k$ and of diameter
$q^{-\deg Q_n-1}$.
\end{lemma}

\section{A transference lemma and the proof of Theorem \ref{kronecker}}

Recall that
$$
M_j(\underline{y})=\sum_{i=1}^n\alpha_{i,j}y_i,\;\;\;\underline{y}=(y_1, \cdots, y_n)^T\;\;\;(1\le j\le m),
$$
are the linear forms determined by the columns of the matrix $A=(\alpha_{i,j})$,
and
$$
L_i(\underline{x})=\sum_{j=1}^m\alpha_{i,j}x_j,\;\;\;\underline{x}=(x_1, \cdots,x_m)^T\;\;\;(1\le i\le n),
$$
are the linear forms determined by its rows.

In this section,
by using a similar method as in the real case (see \cite{cassel}),
we prove a transference lemma,
which establishes a relation between inhomogeneous
simultaneous approximation and homogeneous approximation.
To give the proof, we need some auxiliary results.
We first state a power series analogue of Theorem XVI on page 97 of \cite{cassel}.

\begin{thm}
\label{thm for tran}
Let $l$ be a positive integer and
$f_k(\underline{\theta})$, $g_k(\underline{\xi})$ for $1\le k\le l$ be linear forms in
$\underline{\theta}=(\theta_1, \cdots,\theta_l)$ and $\underline{\xi}=(\xi_1, \cdots,\xi_l)$, respectively.
Suppose that
\begin{equation}
\label{tran lem eq1}
\sum_{k=1}^{l} \, f_k(\underline{\theta})g_k(\underline{\xi})=\sum_{k=1}^{l} \, \theta_k\xi_k
\end{equation}
identically.
Let $\underline{\beta}=(\beta_1, \cdots,\beta_l)$ be a vector in $\mathbb{F}_q((z^{-1}))^l$.
If
\begin{equation}
\label{tran lem eq2}
\left|\left\langle\sum_{k=1}^l  g_k(\underline{\xi})\beta_k\right\rangle\right|
\le \max_{1 \le k \le l} \, \|g_k(\underline{\xi})\|
\end{equation}
holds for all polynomial vectors $\underline{\xi}$,
then there exists a polynomial vector $\underline{b}$ in $\mathbb{F}_q[z]^l$ such that
\begin{equation}
\label{tran lem eq3}
|\langle\beta_k-f_k(\underline{b})\rangle|\le 1, \;\;\; 1\le k\le l.
\end{equation}
\end{thm}

\begin{proof}
We regard $\underline{\xi}$ as a row vector and $\underline{\theta}$, $\underline{\beta}$
as column vectors.
Let $G=(g_{i,j})$ be the ${l\times l}$ square matrix whose $k$-th column is the coefficients of $g_k$
and $F=(f_{i,j})$ be the ${l\times l}$ square matrix whose $k$-th row is the coefficients of $f_k$.
Then, equality (\ref{tran lem eq1}) becomes
\begin{equation*}
(\xi_1,\xi_2,\cdots,\xi_l)
\left(
\begin{array}{cccc}
g_{11} & g_{21} & \cdots & g_{l1} \\
g_{12} & g_{22} & \cdots & g_{l2} \\
\cdots & \cdots & \cdots & \cdots \\
g_{1l} & g_{2l} & \cdots & g_{ll} \\
\end{array}
\right)
\left(
\begin{array}{cccc}
f_{11} & f_{12} & \cdots & f_{1l} \\
f_{21} & f_{22} & \cdots & f_{2l} \\
\cdots & \cdots & \cdots & \cdots \\
f_{l1} & f_{l2} & \cdots & f_{ll} \\
\end{array}
\right)
\left(
\begin{array}{c}
\theta_1 \\
\theta_2 \\
\vdots \\
\theta_l \\
\end{array}
\right)=\sum_{k=1}^l  \theta_k\xi_k.
\end{equation*}
This implies that
\begin{equation}
\label{tran lem eq4}
G=F^{-1}.
\end{equation}
By the analogue of Minkowski's Theorem in $\mathbb{F}_q((z^{-1}))$ proved by Mahler in
Section 9 of \cite{Mahler} and applied to the convex body
$\max_{1\le j\le l}\|g_j(\underline{\xi})\|\le1$,
there is a polynomial $l\times l$ matrix $W$ with $\|\det W\|=1$
whose $k$-th row $\underline{w}^{(k)}$ satisfies
\begin{equation}
\label{tran lem eq5}
\max_{1\le j\le l}\|g_j(\underline{w}^{(k)})\|=\mu_k,\;\;\;\;\prod_{k=1}^l \mu_k=\|\det G\|,
\end{equation}
where the positive real numbers $\mu_k$, $1\le k\le l$, are the successive minima for the function
$\max_{1\le j\le l}\|g_j(\underline{\xi})\|$.

By (\ref{tran lem eq2}), (\ref{tran lem eq5}), and the definition of $g_k(\underline{\xi})$,
we have
\begin{equation*}
\begin{split}
WG\underline{\beta}
&=\left(
\begin{array}{c}
\underline{w}^{(1)}G \\
\underline{w}^{(2)}G \\
\vdots \\
\underline{w}^{(l)}G \\
\end{array}
\right)\underline{\beta}
=\left(
\begin{array}{cccc}
g_1(\underline{w}^{(1)}) & g_2(\underline{w}^{(1)}) & \cdots & g_l(\underline{w}^{(1)}) \\
g_1(\underline{w}^{(2)}) & g_2(\underline{w}^{(2)}) & \cdots & g_l(\underline{w}^{(2)}) \\
\cdots & \cdots & \cdots & \cdots \\
g_1(\underline{w}^{(l)}) & g_2(\underline{w}^{(l)}) & \cdots & g_l(\underline{w}^{(l)})\\
\end{array}
\right)
\left(
\begin{array}{c}
\beta_1 \\
\beta_2 \\
\vdots \\
\beta_l \\
\end{array}
\right)\\
&=
\left(
\begin{array}{c}
\sum_{j=1}^l\beta_j g_j(\underline{w}^{(1)}) \\
\sum_{j=1}^l\beta_j g_j(\underline{w}^{(2)}) \\
\cdots \\
\sum_{j=1}^l\beta_j g_j(\underline{w}^{(l)}) \\
\end{array}
\right)
=\underline{a}+\underline{\delta},
\end{split}
\end{equation*}
where $\underline{a}$ is polynomial vector and
\begin{equation}
\label{tran lem eq6}
\|\underline{\delta}_k\|\le \mu_k\;\;\text{for}\;\;1\le k\le l.
\end{equation}
Hence, by (\ref{tran lem eq4}), we get
\begin{equation}
\label{tran lem eq7}
\underline{\beta}=F\underline{b}+\underline{\gamma},
\end{equation}
where $\underline{b}=W^{-1}\underline{a}$ and $\underline{\delta}=WG\underline{\gamma}$.
Here,  $\underline{b}$ is also a polynomial vector since $\|\det W\|=1$.
By the matrix operation on the ring of matrices whose coordinates are in the fields of power series,
we get
$$
\gamma_j=\frac{\det((WG)_j)}{\det(WG)^{-1}},
$$
where
\begin{equation*}
(WG)_j=
\left(
\begin{array}{ccccccc}
g_1(\underline{w}^{(1)}) & \cdots & g_{j-1}(\underline{w}^{(1)}) & \delta_1 &
g_{j+1}(\underline{w}^{(1)}) & \cdots & g_l(\underline{w}^{(1)}) \\
g_1(\underline{w}^{(2)}) & \cdots & g_{j-1}(\underline{w}^{(2)}) & \delta_2 &
g_{j+1}(\underline{w}^{(2)}) & \cdots & g_l(\underline{w}^{(2)}) \\
\cdots & \cdots & \cdots & \cdots & \cdots & \cdots & \cdots \\
g_1(\underline{w}^{(l)}) & \cdots & g_{j-1}(\underline{w}^{(l)}) & \delta_l &
g_{j+1}(\underline{w}^{(l)}) & \cdots & g_l(\underline{w}^{(l)})\\
\end{array}
\right).
\end{equation*}
By (\ref{tran lem eq5}), the norm of the $k$-th row of the $WG$ is at most $\mu_k$.
Combined with (\ref{tran lem eq6}), we get
\begin{equation}
\begin{split}
\|\gamma_j\|
&\le \|\det G\|^{-1} \, \prod_{k=1}^l \mu_k\le1,
\end{split}
\end{equation}
which gives
$$
|\langle\beta_k-f_k(\underline{b})\rangle|\le 1 \;\;\; (1\le k\le l).
$$
\end{proof}

\begin{cor}
\label{cor for tran}
Let $L_j(\underline{x})$, $M_i(\underline{u})$ be as above and set $l = m+n$.
Let $\underline{\alpha}=(\alpha_1, \cdots,\alpha_n)$ in $\mathbb{F}_q((z^{-1}))^n$,
and $s,t$ be positive integers.
Suppose that
\begin{equation}
\label{tran lem eq8}
|\langle u_1\alpha_1+ \cdots+u_n\alpha_n\rangle|
\le
\max\left\{q^t\max_{1\le i\le m}|\langle M_i(\underline{u})\rangle|,q^{-s}\max_{1\le j\le n}\|u_j\|\right\}
\end{equation}
holds for all polynomial vectors $\underline{u}$.
Then, there exists a polynomial vector $\underline{b}=(b_1, \cdots,b_m)$ with
$$
|\langle L_j(\underline{b})-\alpha_j\rangle|\le q^{-s},\;\;\;\|b_j\|\le q^t, \quad j = 1, \ldots , m.
$$
\end{cor}

\begin{proof}
This is a special case of Theorem \ref{thm for tran}.
Let $C, X$ be in $\mathbb{F}_q((z^{-1}))$ with $\|C\|=q^{-s}$, $\|X\|=q^t$.
Let
\begin{equation*}
\begin{split}
&\underline{\theta}=(\underline{x},\underline{z})=(x_1, \cdots,x_m,z_1, \cdots,z_n),\\
&\underline{\xi}=(\underline{v},\underline{u})=(v_1, \cdots,v_m,u_1, \cdots,u_n),\\
&f_k(\underline{\theta})=\left\{
             \begin{array}{lr}
                C^{-1}(L_k(\underline{x})+z_k), \;\;\;\text{for}\;\;\;k\le n,   \\
                X^{-1}x_{k-n}, \;\;\;\text{for}\;\;\;n<k\le l,
             \end{array}
\right.\\
&g_k(\underline{\xi})=\left\{
             \begin{array}{lr}
                 C u_k,\;\;\;\text{for}\;\;\;k\le n ,  \\
                 X (v_{k-n}-M_{k-n}(\underline{u})), \;\;\;\text{for}\;\;\;n<k\le l,
             \end{array}
\right.
\end{split}
\end{equation*}
and $\underline{\beta}=(C^{-1}\underline{\alpha},\underline{0})$.
The corollary then follows from Theorem \ref{thm for tran}.
\end{proof}

\begin{lemma}[Transference lemma]
\label{tranlemma}
Let $s$ and $t$ be positive integers.
Suppose that the inequality
$$
M(\underline{y})\ge q^{-t}
$$
holds for any non-zero polynomial $n$-tuple $\underline{y}$ of norm $\|\underline{y}\|\le q^s$.
Then, for all $n$-tuples $(\theta_1, \cdots, \theta_n)$ in $ \mathbb{F}_q((z^{-1}))^n$,
there exists a polynomial vector $\underline{x}$ with $\|\underline{x}\|\le q^t$ such that
$$
\max_{1\le i\le n}|\langle L_i(\underline{x})-\theta_i\rangle|\le q^{-s}.
$$
\end{lemma}
\begin{proof}
We apply Corollary \ref{cor for tran} with $\underline{u}=\underline{y}$
and $\underline{\alpha}=\underline{\theta}$.
If $\|\underline{y}\|> q^s$,
then the inequality (\ref{tran lem eq8}) holds,
since the left hand side of inequality (\ref{tran lem eq8}) is not greater than $\frac{1}{q}$.
If $\|\underline{y}\|\le q^s$, then,
since $M(\underline{y})\ge q^{-t}$,
the right hand side of inequality (\ref{tran lem eq8}) is greater than 1
and the inequality (\ref{tran lem eq8}) also holds.
By Corollary \ref{cor for tran},
the proof is established.
\end{proof}


\begin{proof}[Proof of Theorem \ref{kronecker}]
First of all,
we suppose that for every $\eps>0$,
there is a polynomial vector $\underline{x}$ such that simultaneously
  $|\langle L_i(\underline{x})-\theta_i\rangle|\leq\eps$, $(1\le i\le n).$
If $\underline{u}=(u_1,\cdots,u_n)^T$ is any polynomial vector
  such that $A^T\underline{u}$ is in $ \mathbb{F}_q[z]^m$, then
$$
u_1L_1(\underline{x})+\cdots+u_nL_n(\underline{x})=\underline{u}^TA\underline{x}\in\mathbb{F}_q[z].
$$
It follows that
\begin{equation*}
\begin{split}
|\langle u_1\theta_1+\cdots+u_n\theta_n\rangle|
&=|\langle u_1(L_1(\underline{x})-\theta_1)+\cdots+u_n(L_n(\underline{x})-\theta_n)\rangle|\\
&\leq \max\{|\langle u_1(L_1(\underline{x})-\theta_1)\rangle|,\cdots,
|\langle u_n(L_n(\underline{x})-\theta_n)\rangle|\}\\
&\leq \|\underline{u}\|\eps.
\end{split}
\end{equation*}
Since $\eps$ is arbitrary,
we have
$$|\langle u_1\theta_1+\cdots+u_n\theta_n\rangle|=0.$$
Thus
$$u_1\theta_1+\cdots+u_n\theta_n\in \mathbb{F}_q[z].$$

Now we turn to prove that (2) implies (1),
with the help of Corollary \ref{cor for tran}.

For every $\eps>0$,
there is a positive integer $s$ such that $q^{-s}\leq \eps$.

If $|\langle u_1\theta_1+\cdots+u_n\theta_n\rangle|=0$,
then the inequality (\ref{tran lem eq8}) obviously holds.
Otherwise, we have $\max_{1\le i\le m}|\langle M_i(\underline{u})\rangle|>0$ by the assumption.

Since $|\langle u_1\theta_1+\cdots+u_n\theta_n\rangle|\leq q^{-1}$,
inequality (\ref{tran lem eq8}) is satisfied if $\|\underline{u}\|\geq q^s$.
For the finitely many polynomial vectors $\underline{u}$ whose norm is less than $q^s$,
inequality (\ref{tran lem eq8}) still holds if we choose the integer $t$ large enough.
Then the proof is completed by using Corollary \ref{cor for tran}.
\end{proof}


\section{Proof of the Theorem \ref{mainthm}}
We begin by proving that the inequalities
\begin{equation}
\label{proofeq1}
\omega(A,\underline{\theta})\geq  \frac{1}{\widehat{\omega}(A^T)} \;\;\;  \text{and}\;\;\;
\widehat{\omega}(A,\underline{\theta})\geq  \frac{1}{\omega(A^T)}
\end{equation}
hold for all vectors $\underline{\theta}=(\theta_1, \cdots, \theta_n)^T$ in $\mathbb{F}_q((z^{-1}))^n$.

For the first inequality, we can clearly assume that $\widehat{\omega}(A^T)$ is finite.
Let $\omega>\widehat{\omega}(A^T)$ be a real number.
By the definition of the exponent $\widehat{\omega}(A^T)$,
there exists a real number $H$,
which may be chosen arbitrarily large, such that
\begin{equation}
\label{proofeq2}
M(\underline{y})\ge H^{-\omega},
\end{equation}
for any  non-zero polynomial vector $\underline{y}$ of norm at most equal to $H$.
Let $s, t$ be positive integers such that $H^{-\omega}\ge q^{-t}>q^{-1}H^{-\omega}$
and $q^s\le H<q^{s+1}$.
Then we have $M(\underline{y})\ge H^{-\omega}\ge q^{-t}$
for any  non-zero polynomial vector $\underline{y}$ of norm at most equal to $q^s$.
By Lemma \ref{tranlemma},
there exists a polynomial $n$-tuple $\underline{x}$ with $\|\underline{x}\|\le q^t$ such that
$$
\max_{1\le i\le n}|\langle L_i(\underline{x})-\theta_i\rangle|\le q^{-s}\le qH^{-1}
<q^{1+\frac{1}{\omega}}q^{-t\frac{1}{\omega}}
<q^{1+\frac{1}{\omega}}\|\underline{x}\|^{-\frac{1}{\omega}}.
$$
This shows that $\omega(A,\underline{\theta})\geq  \frac{1}{\omega}$.

For the second inequality of (\ref{proofeq1}),
we can clearly assume that ${\omega}(A^T)$ is finite.
For $\omega>\omega(A^T)$ and all real number $H$ with sufficiently large,
the inequality (\ref{proofeq2}) is satisfied for
any non-zero polynomial vector $\underline{y}$ of norm $\|\underline{y}\|\le H$.
We argue in a similar way as in the proof of the first inequality.
We omit the details.

We now prove that
\begin{equation}
\label{proofeq3}
\omega(A,\underline{\theta})\leq  \frac{1}{\widehat{\omega}(A^T)} \;\;\;  \text{and}\;\;\;
\widehat{\omega}(A,\underline{\theta})\leq  \frac{1}{\omega(A^T)}
\end{equation}
hold for almost all vectors $\underline{\theta}=(\theta_1, \cdots, \theta_n)^T$
in $\mathbb{F}_q((z^{-1}))^n$.

By the formula $\underline{y}^TA\underline{x}=\underline{x}^TA^T\underline{y}$,
it is easily seen that
$$
y_1\theta_1+ \cdots + y_n\theta_n=
\sum_{j=1}^mx_jM_j(y_1,\cdots,y_n)-\sum_{i=1}^ny_i(L_i(x_1,\cdots,x_m)-\theta_i),
$$
from which it follows that
\begin{equation}
\label{proofeq4}
|\langle y_1\theta_1+\cdots+y_n\theta_n\rangle|\le
\max\left\{\|\underline{y}\|\max_{1\le i\le n}|\langle L_i(\underline{x})-\theta_i\rangle|,
\|\underline{x}\|M(\underline{y})\right\},
\end{equation}
for all polynomial vectors $\underline{x}=(x_1,\cdots,x_m)^T$ and $\underline{y}=(y_1,\cdots,y_n)^T$.

We follow the notations in Section 3 and denote by
$$
\underline{y}_i=(y_{i1}, \cdots,y_{in})^T\;\;\;\text{and}\;\;\;Y_i=\|\underline{y}_i\|\;\;\;(i\ge1)
$$
the sequence of best approximations associated with the matrix $A^T$.

By Lemma \ref{bestlemma},
for almost all $\underline{\theta}$ in $\mathbb{F}_q((z^{-1}))^n$,
the inequality
\begin{equation}
\label{proofeq5}
|\langle y_{i1}\theta_1+ \cdots+y_{in}\theta_n\rangle|\ge Y_i^{-\delta}
\end{equation}
holds for all $\delta>0$ and any index $i$ large enough.
Let us fix two real numbers $\delta$ and $\omega$ such that
$$
0<\delta<\omega<\widehat{\omega}(A^T).
$$
Let $\underline{x}$ be a polynomial $m$-tuple with sufficiently large norm $\|\underline{x}\|$,
and let $k$ be the index defined by the inequality
$$
Y_k\le \|\underline{x}\|^{\frac{1}{\omega-\delta}}<Y_{k+1}.
$$
This gives
$$
Y_{k+1}^{\omega}> \|\underline{x}\|^{\frac{\omega}{\omega-\delta}}\ge \|\underline{x}\|Y_k^{\delta}.
$$

By (iii) of Lemma \ref{property of best appro},
we have
$$
\|\underline{x}\|M(\underline{y}_k)\le \|\underline{x}\|Y_{k+1}^{-\omega}< Y_k^{-\delta}.
$$
Using (\ref{proofeq4}) with $\underline{y}=\underline{y}_k$  and (\ref{proofeq5}) with $i=k$,
we deduce that
\begin{equation*}
Y_k^{-\delta}
\leq \|\underline{y}_k\|\max_{1\le i\le n}|\langle L_i(\underline{x})-\theta_i\rangle|
\leq Y_k\max_{1\le i\le n}|\langle L_i(\underline{x})-\theta_i\rangle|,
\end{equation*}
which gives
$$
|\langle A\underline{x}-\underline{\theta}\rangle|
= \max_{1\le i\le n}|\langle L_i(\underline{x})-\theta_i\rangle|
\geq Y_k^{-1-\delta}
\geq \|\underline{x}\|^{-\frac{1+\delta}{\omega-\delta}}.
$$
This implies
$$
\omega(A,\underline{\theta})\le \frac{1+\delta}{\omega-\delta}.
$$
Let $\delta$ and $\omega$ be arbitrarily close to  $0$ and to $\widehat{\omega}(A^T)$, respectively.
Then, it is immediate that the first inequality of (\ref{proofeq3}) holds.

The second upper bound can be handled in the same manner.
Let us fix now two real numbers $\delta$ and $\omega$ such that
$$
0<\delta<\omega<\omega(A^T).
$$
Let $\underline{x}$ be a polynomial $m$-tuple with $\|\underline{x}\|\le H_k
:= {Y_k^{\omega-\delta}}/{2}$.
By (iv) of Lemma \ref{property of best appro},
there exist infinitely many integers $k\ge 1$
such that $M(\underline{y}_k) \le Y_k^{-\omega}$,
thus, for which,
$$
\|\underline{x}\|M(\underline{y}_k)
\leq \|\underline{x}\|Y_{k}^{-\omega}
\leq  \frac{Y_k^{-\delta}}{2}.
$$
Applying again inequality (\ref{proofeq4}),
we obtain
$$
Y_k^{-\delta}
\le \|\underline{y}_k\|\max_{1\le i\le n}|\langle L_i(\underline{x})-\theta_i\rangle|
\le Y_k\max_{1\le i\le n}|\langle L_i(\underline{x})-\theta_i\rangle|,
$$
which yields
$$
|\langle A\underline{x}-\underline{\theta}\rangle|
= \max_{1\le i\le n}|\langle L_i(\underline{x})-\theta_i\rangle|
\ge Y_k^{-1-\delta}
= 2^{-\frac{1+\delta}{\omega-\delta}}H_k^{-\frac{1+\delta}{\omega-\delta}}.
$$
Since the above lower bound holds for any polynomial $\underline{x}$
whose norm is less than $H_k$
and for infinitely many $k\ge 1$,
noting that the sequence $(H_i)_{i \ge 1}$ tends to infinity,
it follows that
$$
\widehat{\omega}(A,\underline{\theta})\le \frac{1+\delta}{\omega-\delta}.
$$
Choosing $\delta$ and $\omega$ arbitrarily close to  $0$ and to $\omega(A^T)$ respectively,
we get the second inequality of (\ref{proofeq3}),
and the proof of first assertion is completed.

\medskip

It only remains to prove that
$$
\widehat{\omega}(A,\underline{\theta}) \le \omega(A),
$$
when $\underline{\theta} =(\theta_1, \cdots, \theta_n)^T$ is not in
$A\mathbb{F}_q[z]^m+\mathbb{F}_q[z]^n$.

For any $\underline{x}$ in $\mathbb{F}_q[z]^m$,
set $L(\underline{x})=|\langle A\underline{x}-\underline{\theta}\rangle|$.
By the denseness of $A\mathbb{F}_q[z]^m$ in $\mathbb{F}_q((z^{-1}))^n$
(which is implied by Theorem \ref{kronecker}) and following
the same method as in the homogeneous case,
we can construct a sequence of polynomial vectors $\underline{x}_i$,
$i\ge1$, in $ \mathbb{F}_q[z]^m$
associated with $L(\underline{x}_1),L(\underline{x}_2),\cdots$
which satisfy the following properties.
Set $\|\underline{x}_i\|=H_i$ and $L_i=L(\underline{x}_i)$,
then we have
$$
1=H_1<H_2<\cdots\;\;\;\text{and}\;\;\;L_1>L_2>\cdots
$$
and $L(\underline{x})\ge L_i$ for all polynomial vectors $\underline{x}$ with $\|\underline{x}\|<H_{i+1}$.
Here we also call the above sequence $(\underline{x}_i)_{i\geq1}$
a sequence of best approximations related to
$L_1,L_2,\cdots$.
By definition of $\widehat{\omega}(A,\underline{\theta})$ and best approximation,
for any $\omega<\widehat{\omega}(A,\underline{\theta})$,
the inequality
$$
0<|\langle A\underline{x}_i-\underline{\theta}\rangle|\le H_{i+1}^{-\omega}
$$
holds for any index $i$ sufficiently large in terms of $\omega$.
By using the triangle inequality,
we conclude that
\begin{equation*}
\begin{split}
|\langle A(\underline{x}_i-\underline{x}_{i-1})\rangle|
&=|\langle A\underline{x}_i-\underline{\theta}-(A\underline{x}_{i-1}-\underline{\theta})\rangle|\\
&\le \max\left\{|\langle A\underline{x}_i-\underline{\theta}\rangle|,
|\langle A\underline{x}_{i-1}-\underline{\theta}\rangle|\right\}\\
&\le H_i^{-\omega},
\end{split}
\end{equation*}
which gives that $\omega(A)\ge \omega$.
Choosing $\omega$ arbitrarily close to $\widehat{\omega}(A,\underline{\theta})$,
we complete the proof.


\section{Proof of Theorem \ref{prop}}

Before proving Theorem \ref{prop} we establish an auxiliary lemma.

\begin{lemma}
\label{dimension}
Let $l \ge 2$ be an integer.
For a sequence $(\underline{h}_k)_{k\ge1}$ of polynomial vectors such that
$\|\underline{h}_k\|\geq q^l\|\underline{h}_{k-1}\|$ for $k\geq 2$,   
set
$$
S_{\{\underline{h}_k\}}=\{\underline{\theta}\in \mathbb{I}^n\colon
\; \text{there exists}\;\;  k_0(\underline{\theta})\;\; \text{such that}\;\;
|\langle\underline{h}_k\underline{\theta}\rangle| \geq q^{-1}
\;\;\;\text{for all}\;\;\;k\ge k_0(\underline{\theta})\}.
$$
Then we have $\dim_H S_{\{\underline{h}_k\}}\ge n-\frac{1}{l}$.
\end{lemma}

\begin{proof}
Our strategy to prove this lemma is as follows.
First,
we define some partitions of $\mathbb{I}^n$ and construct a family of balls covering the points
which do not satisfy the condition in the definition of the set $S_{\{\underline{h}_k\}}$.
Then we delete the family of balls from the partitions to construct a
Cantor subset contained in $S_{\{\underline{h}_k\}}$.

For any $i\geq1$, define $d_i$ by
$\|\underline{h}_i\|=q^{d_i}$ and set
$$
\Gamma_i=z^{-d_i-1}\mathbb{F}_q[z]^n\cap \mathbb{I}^n.
$$
It is clear that all distinct elements $\underline{x},\underline{y}$ in $\Gamma_i$ satisfy
\begin{equation}
\label{lattice eq1}
\|\underline{x}-\underline{y}\|\ge q^{-d_i-1}.
\end{equation}
Now we define a partition of $\mathbb{I}^n$.
For each $i\ge1$,
let $\mathscr{C}_i$ be the family of balls $B(\underline{c},q^{-d_i-1})$ centered
at some point $\underline{c}$ in $\Gamma_i$,
i.e.,
$$
\mathscr{C}_i=\{B(\underline{c},q^{-d_i-1})\colon \underline{c}\in \Gamma_i\}.
$$
By (\ref{lattice eq1}) and the ball intersection property,
any two distinct balls in $\mathscr{C}_i$ have empty intersection.
Each ball in $\mathscr{C}_i$ has measure $q^{(-d_i-1)n}$.
Since there are exactly $q^{(d_i+1)n}$ of these balls,
they do indeed define a partition of $\mathbb{I}^n$.

For any $i\ge 1$,
we consider the resonant set
$$
R_i=\{\underline{x}\in \mathbb{I}^n\colon \underline{h}_i\underline{x}
=p \;\;\;\text{for some}\;\;\;p\in \mathbb{F}_q[z]\}.
$$
Since $\underline{x}$ is in $\mathbb{I}^n$,
each resonant set $R_i$ is contained in one of the affine spaces
$$
R_i(r)=\{\underline{x}\in \mathbb{I}^n\colon \underline{h}_i\underline{x}=r\},
\;\;\;\text{where}\;r \;\text{is in}\;\; \mathbb{F}_q[z]
\;\;\text{with}\;\;\|r\|\leq \|\underline{h}_i\|.
$$

In each $R_i(r)$,
we choose a subset $\Lambda_i(r)$ such that the distance between
any two different points in $\Lambda_i(r)$ is at least $q^{-d_i-1}$
and such that, for any point $\underline{\xi}$ in $R_i(r)$,
there is a point $\underline{\eta}$ in $\Lambda_i(r)$ at a distance to $\underline{\xi}$
less than $q^{-d_i-1}$.
Let $\Lambda_i$ be the union of the sets $\Lambda_i(r)$ where $\|r\|\leq \|\underline{h}_i\|$.
Set
$$
\mathcal{G}_i=\{B(\underline{c},q^{-d_i-1})\colon \underline{c}\in \Lambda_i\}.
$$
If $\underline{\theta}$ in $\mathbb{I}^n$ satisfies
$|\langle \underline{h}_i\underline{\theta}\rangle|<\frac{1}{q}$, then we have
$$
\|\underline{h}_i\|\text{dist}_{\infty}(\underline{\theta},R_i)
\leq |\langle \underline{h}_i\underline{\theta}\rangle|
<\frac{1}{q},
$$
where $\text{dist}_{\infty}$ denotes the distance associated with the supremum norm.
Then,
$$
\text{dist}_{\infty}(\underline{\theta},R_i)< q^{-d_i-1}.
$$
which implies that there exists $\underline{\xi}$ in $R_i$ such that
$$
\|\underline{\theta}-\underline{\xi}\|<q^{-d_i-1},
$$
and, consequently,
 $\underline{\theta}$ is contained in some ball which belongs to $\mathcal{G}_i$.

Let $\mathcal{D}_i=\{B\in \mathscr{C}_i \colon B\cap \mathcal{G}_i=\emptyset\}$.
Define
$$
E_i=\bigcup_{B\in \mathcal{D}_i}B \;\;\;\text{and}\;\;\;E=\bigcap_{i=1}^{\infty}E_i.
$$
Then, $E\subset S_{\{\underline{h}_k\}}$.

Now we determine the Hausdorff dimension of the set $E$.
By the ball intersection property, the distance between
any two balls in $\mathcal{D}_i$ is $\epsilon_i=q^{-d_i-1}$.
Since $\mathscr{C}_i$ is a partition of $\mathbb{I}^n$,
for any ball $B$ in $\mathcal{D}_i$,
the number of balls of $\mathscr{C}_{i+1}$ contained in $B$ is $q^{(d_{i+1}-d_i)n}$.

For any $\underline{\xi}$ in $R_{i+1}(r)$, $\underline{\theta}$ in $R_{i+1}(t)$,
where $r,t$ are in $ \mathbb{F}_q[z]$, we obtain
$$
1\le \|r-t\|\le\|\underline{h}_{i+1}\underline{\xi}-\underline{h}_{i+1}\underline{\theta}\|\le
\|\underline{h}_{i+1}\|\|\underline{\xi}-\underline{\theta}\|,
$$
hence
$$
\|\underline{\xi}-\underline{\theta}\|\geq \frac{1}{\|\underline{h}_{i+1}\|}.
$$
Consequently, the number of affine spaces
which can intersect a ball $B$ in $\mathcal{D}_i$ is at most $q^{d_{i+1}-d_i-1}$.
Since every such affine space contains $q^{(d_{i+1}-d_i)(n-1)}$ points of $\Lambda_{i+1} \cap B$,
the number of balls of $\mathcal{D}_{i+1}$ contained in the ball $B$ is at least
$$
m_{i+1}=q^{(d_{i+1}-d_i)n}-q^{(d_{i+1}-d_i)n-1}=q^{(d_{i+1}-d_i)n}(1-q^{-1})
\geq 2^{-1}q^{(d_{i+1}-d_i)n}.
$$
Since $\|\underline{h}_k\|\geq q^l\|\underline{h}_{k-1}\|$ for $k\geq 2$,
we have $d_k\geq (k-1) l$.
By this fact and Example 4.6 of \cite{Falconer}, we have
\begin{equation*}
\begin{split}
\dim_H E
&\geq \liminf_{k\to+\infty}\frac{\log m_1m_2\cdots m_{k-1}}{-\log m_k \epsilon_k^n}\cdot n\\
&\geq \liminf_{k\to+\infty}\frac{k\log\frac{1}{2}+nd_{k-1}\log q}{-\log \frac{1}{2}+n(d_{k-1}+1)\log q}
\cdot n\\
&\geq \liminf_{k\to+\infty}\frac{nd_{k-1}-k}{n(d_{k-1}+2)}\cdot n\\
&\geq \liminf_{k\to+\infty}\frac{nd_{k-1}-\frac{d_{k-1}}{l}-2}{n(d_{k-1}+2)}\cdot n\ge n-\frac{1}{l}.
\end{split}
\end{equation*}
The proof is complete.

\end{proof}

Now we prove Theorem \ref{prop}.

For a positive integer $l\geq2$,
we extract a subsequence $(\underline{y}_{\varphi_l(k)})_{k\ge1}$ from
the sequence of best approximations $(\underline{y}_k)_{k\ge1}$,
where the index function is an increasing function $\varphi_l:\mathbb{Z}_{\ge1}\to\mathbb{Z}_{\ge1}$
satisfying $\varphi_l(1)=1$ and,
for any integer $i\ge 2$,
\begin{equation}
\label{prop proofeq1}
Y_{\varphi_l(i)}\geq q^lY_{\varphi_l(i-1)}\;\;\;\text{and}\;\;\;
Y_{\varphi_l(i-1)+1}\geq q^{-2l}Y_{\varphi_l(i)}.
\end{equation}

Let
$$
\mathcal{J}_0=\{j\colon Y_{j+1}\geq q^lY_j\}.
$$
To define the function $\varphi_l$ we distinguish two cases,
according to whether the set $\mathcal{J}_0$ is finite or not.

If $\mathcal{J}_0$ is an infinite set,
then set $\varphi_l(1)=1$.
Suppose that $\varphi_l(i)$ has already been defined for $1\le i\le h'$,
and define $\varphi_l(h)$ to be the smallest element of $\mathcal{J}_0$ greater than $\varphi_l(h')$.
We let $\varphi_l(h-1)$ be the largest index $t\ge \varphi_l(h')$ for which $Y_{\varphi_l(h)}\ge q^lY_t$,
we let $\varphi_l(h-2)$ be the largest index $t\ge \varphi_l(h')$
for which $Y_{\varphi_l(h-1)}\ge q^lY_t$,
and so on until an index $t$ as above does not exist.
We have just defined $\varphi_l(h),\varphi_l(h-1),\cdots,\varphi_l(h-h_0)$.
Then, we set $h=h'+h_0+1$,
and the inequalities (\ref{prop proofeq1}) are satisfied for
$i=h'+1,\cdots,h'+h_0+1$.

If $\mathcal{J}_0$ is a finite set,
we denote by $g$ the largest of its elements, putting $g=1$ if $\mathcal{J}_0$ is empty.
We apply the above process to construct the initial values of the function $\varphi$
up to $g=\varphi_l(h)$.
Then, we define $\varphi_l(h+1)$ as the smallest index $t$ for which $Y_t\geq q^lY_{\varphi_l(h)}$.
We observe that $Y_{\varphi_l(h+1)-1}<q^l Y_{\varphi_l(h)}$ and
$Y_{\varphi_l(h)+1}\geq Y_{\varphi_l(h)}>q^{-l} Y_{\varphi_l(h+1)-1}>q^{-2l} Y_{\varphi_l(h+1)}$,
as required.
We continue in this way,
by defining $\varphi_l(h+2)$ as the smallest index $t$ for which $Y_t\geq q^l Y_{\varphi_l(h+1)}$,
and so on.
The inequalities  (\ref{prop proofeq1}) are then satisfied.

By Lemma \ref{dimension},
for any $\underline{\theta}$ in $S_{\{y_{\underline{\varphi}_l(i)}\}}$, it follows that
$$
|\langle y_{\varphi_l(i),1}\theta_1 +\cdots+y_{\varphi_l(i),n}\theta_n\rangle|\geq \frac{1}{q},
\;\text{for sufficiently large}\; i.
$$
Let $\underline{x}$ be a non-zero polynomial $m$-tuple whose norm is sufficiently large and
let $k$ be the index defined by the inequalities
$$
Y_{\varphi_l(k)}\leq q^{(2l+1)}q^{\frac{m}{n}}\|\underline{x}\|^{\frac{m}{n}}<Y_{\varphi_l(k+1)}.
$$
By Lemma \ref{property of best appro} and inequality (\ref{proofeq4})
with $\underline{y}=y_{\varphi_l(k)}$,
we have
$$
\frac{1}{q}
\leq \max\left\{q^{(2l+1)}q^{\frac{m}{n}}\|\underline{x}\|^{\frac{m}{n}}
|\langle A\underline{x}-\underline{\theta}\rangle|,
\|\underline{x}\|q^{\frac{n}{m}}Y_{\varphi_l(k)+1}^{-\frac{n}{m}}\right\}.
$$
By construction of the subsequence $(Y_{\varphi_l(i)})_{i\geq1}$,
we have $Y_{\varphi_l(k)+1}^{-1}Y_{\varphi_l(k+1)}\leq q^{2l}$,
so
$$
\|\underline{x}\|q^{\frac{n}{m}}Y_{\varphi_l (k)+1}^{-\frac{n}{m}}
< q^{-1}q^{-\frac{(2l+1)n}{m}}q^{\frac{n}{m}}q^{\frac{2ln}{m}}
=q^{-1},
$$
then
$$
\frac{1}{q}\leq q^{(2l+1)}q^{\frac{m}{n}}\|\underline{x}\|^{\frac{m}{n}}
|\langle A\underline{x}-\underline{\theta}\rangle|,
$$
which gives
$$
|\langle A\underline{x}-\underline{\theta}\rangle|\geq q^{-(2l+2)}
q^{-\frac{m}{n}}\|\underline{x}\|^{-\frac{m}{n}}.
$$
From this, we deduce that $S_{\{\underline{y}_{\varphi_l(i)}\}}\subset \Bad^\eps (A)$
with $\eps=q^{-(2l+2)}q^{-\frac{m}{n}}$, and then
$$
\dim_H \Bad^\eps (A)\geq n-\frac{1}{l},
$$
which implies the second assertion.

Recall that
$$
\Bad(A) := \underset{\eps>0}{\bigcup} \, \Bad^\eps (A)
= \{\underline{\theta}\in
\mathbb{I}^n :  \liminf_{\underline{x}\in \mathbb{F}_q[z]^m,\|\underline{x}\|\to \infty}
\|\underline{x}\|^{m/n} \cdot |\langle A \underline{x} - \underline{\theta} \rangle|  > 0  \}.
$$
We have just proved that, for any integer $l \ge 2$, we have
$$
S_{\{\underline{y}_{\varphi_l(i)}\}}\subset {\Bad}(A).
$$
Letting $l$ tend to infinity,
we obtain
$$
\dim_H {\Bad}(A)=n.
$$
This completes the proof of the theorem.


\section{Proof of Theorem \ref{higherdim}}

We use the same method as in the last section.
The next lemma can be seen as a sharpening of
Lemma \ref{dimension} when the sequence of norms of
the polynomial vectors increase very rapidly.

\begin{lemma}
\label{higherdimenlemma}
For any $\delta$ in $ (0, q^{-1}]$,
let $(\underline{h}_k)_{k\geq1}$  be a sequence  of polynomial vectors such that
$\frac{\|\underline{h}_{k+1}\|}{\|\underline{h}_k\|} \ge q\delta^{-1}$ for $k\geq1$
and $\lim_{k\to\infty}\|\underline{h}_k\|^{\frac{1}{k}}=\infty$.
Then, the set
$$
S_{\delta}=\{\underline{\theta}\in \mathbb{I}^n\colon\;\;\;
\text{there exists}\;\;\; k_0(\underline{\theta}) \;\;\text{such that}\;\;\;
|\langle\underline{h}_k\underline{\theta}\rangle|\geq\delta
\;\;\text{for all}\;\;k\geq k_0(\underline{\theta})\}
$$
has full Hausdorff dimension.
\end{lemma}

\begin{proof}
Since the proof is very similar to that of Lemma \ref{dimension},
we just give the necessary modifications here.

Let $\delta$ be in $ (0, q^{-1}]$.
For any $k \geq1$,
set $\|\underline{h}_k\|=q^{d_k}$.
We note that $\delta$ plays the role of $q^{-1}$ in the proof of Lemma \ref{dimension}.
The remaining part of the construction of a suitable subset can be done in a similar way.
Notice that, since $\frac{d_k}{k}$ tends to infinity with $k$,
we have
\begin{equation*}
\begin{split}
\dim_H E
&\geq \liminf_{k\to+\infty}\frac{\log m_1m_2\cdots m_{k-1}}{-\log m_k \varepsilon_k^n}\cdot n\\
&=\liminf_{k\to+\infty}\frac{k\log\frac{1}{2}+nd_{k-1}\log q}{-\log \frac{1}{2}+n(d_{k-1}+1)\log q}
\cdot n=n,
\end{split}
\end{equation*}
which completes the proof.

\end{proof}

Let us begin the proof of Theorem \ref{higherdim}.

Let
$$
\underline{y}_k=(y_{k1}, \cdots,y_{kn})^T,\;\;\;k\geq1,
$$
be the sequence of best approximations associated to the matrix $A^T$,
and set $Y_k:=\|\underline{y}_k\|$ for $k\geq1$.

Let $\delta$ be in $ (0, q^{-1}]$ and set $R=q\delta^{-1}$.
Since $Y_k^{\frac{1}{k}}$ tends to infinity with $k$,
the set
$$
\mathcal{J}_R=\{j\colon Y_{j+1}\geq RY_j\}.
$$
is an infinite set.
In the same way as in the proof of Theorem \ref{prop},
we can extract a subsequence $(\underline{y}_{\varphi(k)})_{k\geq1}$
of $(\underline{y}_k)_{k\geq1}$ with the property that
\begin{equation}
\label{subsequence}
Y_{\varphi(k)}\geq RY_{\varphi(k-1)},\;\;Y_{\varphi(k-1)+1}
\geq R^{-1}
Y_{\varphi(k)}, \;\;\text{for}\;\;k\geq2.
\end{equation}
We apply Lemma \ref{higherdimenlemma} to
$(\underline{y}_{\varphi(k)})_{k\geq1}$ and take
$\underline{\theta} = (\theta_1, \ldots , \theta_n)$ in the corresponding set $S_{\delta}$,
that is, satisfying
\begin{equation}
\label{neq6}
|\langle y_{\varphi(k)1}\theta_1+ \cdots+y_{\varphi(k)n}\theta_n\rangle|
\ge \delta,\;\text{for sufficiently large}\;k.
\end{equation}
Let $\underline{h}$ be a non-zero polynomial $m$-tuple whose norm is sufficiently large
and let $k$ be the index defined by the inequality
$$
Y_{\varphi(k)}\le qR\delta^{-\frac{m}{n}}\|\underline{h}\|^{\frac{m}{n}}<Y_{\varphi(k+1)}.
$$

By (\ref{proofeq4}), (\ref{subsequence}), and (ii) of Lemma \ref{property of best appro} with
$\underline{y}=\underline{y}_{\varphi(k)}$ and $\underline{x}=\underline{h}$,
since
$$
\|\underline{h}\|M(\underline{y}_{\varphi(k)})
\leq \|\underline{h}\|q^{\frac{n}{m}}Y_{\varphi(k)+1}^{-\frac{n}{m}}
< \delta(qR)^{-\frac{n}{m}}
q^{\frac{n}{m}}Y_{\varphi(k)+1}^{-\frac{n}{m}}Y_{\varphi(k+1)}^{\frac{n}{m}}\leq \delta,
$$
we have
$$
\delta
\leq Y_{\varphi(k)}|\langle A\underline{h}-\underline{\theta}\rangle|
\leq qR\delta^{-\frac{m}{n}}\|\underline{h}\|^{\frac{m}{n}}|\langle A\underline{h}
-\underline{\theta}\rangle|.
$$
Consequently, we get
$$
\|\underline{h}\|^{\frac{m}{n}}|\langle A\underline{h}-\underline{\theta}\rangle|
\ge\frac{\delta^{1+\frac{m}{n}}}{qR}
=\frac{\delta^{2+\frac{m}{n}}}{q^2}.
$$
By letting $\delta=q^{-1}$,
this gives the first assertion of Theorem \ref{higherdim}.

If $m=n=1$, $A = (\alpha)$, and the degrees of the partial quotients of $\alpha$ tend to infinity,
then the assumption of Lemma \ref{higherdimenlemma}
is satisfied for $h_k=Q_{k+N}$ for some constant $N\geq 0$.
For any $0<\delta\le\frac{1}{q}$,
the set $S_{\delta}$ has full Hausdorff dimension.
Let $x$ be in $\mathbb{I}$, let $h$ be a polynomial. Then,
for every $y$ in $\mathbb{F}_q[z]$, we have
\begin{equation}
\label{neqonedim}
|\langle yx\rangle|=|\langle yx-y\alpha h+y\alpha h\rangle|
\le \max\{\|y\||\langle h\alpha-x\rangle|,\|h\||\langle y\alpha\rangle|\}.
\end{equation}

Now we assume that $\|h\|$ is large enough and let $l$
be the integer with $\|Q_l\|\le\delta^{-1}\|h\|<\|Q_{l+1}\|$.
For any $\theta$ in $S_{\delta}$,
letting $y=Q_l$ and $x=\theta$ in the inequality (\ref{neqonedim}),
since
$\|h\|\|\{Q_l\alpha\}\|=\frac{\|h\|}{\|Q_{l+1}\|}<\delta$,
we have
$$
\delta\le|\langle Q_l\theta\rangle|\le \|Q_l\||\langle h\alpha-\theta\rangle|
\le \delta^{-1}\|h\||\langle h\alpha-\theta\rangle|.
$$
This gives $\|h\||\langle h\alpha-\theta\rangle|\geq \delta^2$.
Setting $\delta=\frac{1}{q}$,
the proof is complete.

\section{Proof of Theorem \ref{onedimthm}}

Since we always have $\omega((\xi))=1$ for any irrational power series $\xi$
whose partial quotients have bounded degree,
we may assume that $\omega>1$.

If $\omega((\xi))$ is finite and equal to $\omega$, then let
let $(\omega_n)_{n\ge0}$ be the constant sequence equal to $\omega$,
otherwise, put $\omega_n=n$ for any $n\ge0$.
Let $\xi$ be an element in $\mathbb{F}_q((z^{-1}))$
such that the sequence of the denominators $(Q_n)_{n\ge0}$
of its convergents $\frac{P_n}{Q_n}$ satisfies the growth condition
$$
\|Q_n\|^{\omega_n}\le \|Q_{n+1}\|<q\|Q_n\|^{\omega_n}.
$$
By Theorem \ref{mainthm},
we have $\widehat{\omega}((\xi),\theta)=\frac{1}{\omega((\xi))}$
for almost all $\theta$ in $\mathbb{F}_q((z^{-1}))$.
Let $\nu$ be a non-negative real number.
If $\omega((\xi))$ is finite, then assume furthermore that $\frac{1}{\omega} \le \nu\le \omega$.
We construct an element $\theta$ in $\mathbb{F}_q((z^{-1}))$
for which $\widehat{\omega}((\xi),\theta)=\nu$.
When $\omega((\xi))=+\infty$,
our process furnishes moreover some $\theta$ not in $\mathbb{F}_q[z]+\xi\mathbb{F}_q[z]$
with  $\widehat{\omega}((\xi),\theta)=+\infty$.

Let $(u_n)_{n\ge0}$ be a sequence of polynomials with
$$
\|Q_n\|^{\frac{\omega_n-\nu}{\nu+1}}\le \|u_n\|<q\|Q_n\|^{\frac{\omega_n-\nu}{\nu+1}}, \quad
\hbox{for $n \ge 1$}.
$$
Set
$$
\theta=\sum_{k\ge0}u_k(Q_k\xi-P_k).
$$
For any $n\ge0$, set
$$
V_n=\sum_{k = 0}^n u_kQ_k\;\;\;\text{and}\;\;\;W_n=\sum_{k=0}^n u_kP_k.
$$
Then we have
$$
\|V_n\|=\|u_n\|\|Q_n\|\;\;\;\text{and}\;\;\;\|V_n\xi-W_n-\theta\|=\|u_{n+1}\|\|Q_{n+2}\|^{-1},
$$
so
$$
\|Q_n\|^{\frac{\omega_n+1}{\nu+1}}\le \|V_n\|<q\|Q_n\|^{\frac{\omega_n+1}{\nu+1}}
$$
and
$$
q^{-1}\|Q_{n+1}\|^{-\frac{\nu(\omega_{n+1}+1)}{\nu+1}}
<\|V_n\xi-W_n-\theta\|<q\|Q_{n+1}\|^{-\frac{\nu(\omega_{n+1}+1)}{\nu+1}},
$$
hence
\begin{equation}
\label{oneproofeq1}
\|V_n\xi-W_n-\theta\|<q\|Q_{n+1}\|^{-\frac{\nu(\omega_{n+1}+1)}{\nu+1}}\le q^{1+\nu}\|V_{n+1}\|^{-\nu}
\end{equation}
which implies that $\widehat{\omega}((\xi),\theta)\ge \nu$.
When $\omega ((\xi)) =+\infty$,
we construct $\theta$ in $ \mathbb{F}_q((z^{-1}))$ not in $\mathbb{F}_q[z]+\xi\mathbb{F}_q[z]$
and with $\widehat{\omega}((\xi),\theta)=+\infty$
exactly in the same way,
by taking $u_n=1$ for any $n\ge0$.

Next we prove that for infinitely many $n$ and all
polynomials $x$ and $y$ with $\|x\|\le\frac{1}{q}\|V_n\|$,
we have
\begin{equation}
\label{oneproofeq2}
\|x\xi-y-\theta\|\ge  q^{-2} \|V_n\|^{-\nu}.
\end{equation}
It follows that $\widehat{\omega}((\xi),\theta)\le \nu$,
and therefore that $\widehat{\omega}((\xi),\theta)=\nu$.

To obtain a contradiction,
we suppose inequality (\ref{oneproofeq2}) does not hold
for some polynomials $x$ and $y$ with $\|x\|\le\frac{1}{q}\|V_n\|$.
Then we deduce from (\ref{oneproofeq1}) and the triangle inequality that
\begin{equation*}
\begin{split}
\|(x-V_{n-1})\xi-(y-W_{n-1})\|
&= \|x\xi-y-\theta-(V_{n-1}\xi-W_{n-1}-\theta)\|\\
&\le \max\{\|x\xi-y-\theta\|,\|V_{n-1}\xi-W_{n-1}-\theta\|\}\\
&\le  q^{1+\nu}\|V_n\|^{-\nu}.
\end{split}
\end{equation*}

Set
$a=-P_n(x-V_{n-1})+Q_n(y-W_{n-1})$
and
$b=P_{n-1}(x-V_{n-1})-Q_{n-1}(y-W_{n-1})$,
if $n$ is even (the case $n$ is odd can be handled in the same way).
Then we have
$$
x-V_{n-1}=aQ_{n-1}+bQ_n\;\;\;\text{and}\;\;\;y-W_{n-1}=aP_{n-1}+bP_n.
$$
A trivial verification shows that
\begin{equation*}
\begin{split}
b
&=(x-V_{n-1})P_{n-1}-Q_{n-1}(y-W_{n-1})\\
&=(x-V_{n-1})(P_{n-1}-\xi Q_{n-1})-Q_{n-1}(y-W_{n-1}-(x-V_{n-1})\xi).
\end{split}
\end{equation*}
This gives
\begin{equation*}
\begin{split}
\|b\|
&\le \max \{q^{1+\nu}\|Q_{n-1}\|\|V_n\|^{-\nu},q^{-1}\|V_n\|\|Q_n\|^{-1}\}\\
&=q^{-1}\|V_n\|\|Q_n\|^{-1}\le q^{-1}\|u_n\|.
\end{split}
\end{equation*}
Now we use the formula
$$
x\xi-y-\theta=a(Q_{n-1}\xi-P_{n-1})-(u_n-b)(Q_n\xi-P_n)-\sum_{k\ge n+1}u_k(Q_k\xi-P_k).
$$
When $a\neq 0$,
we bound from below
$$
\|x\xi-y-\theta\|=\|a(Q_{n-1}\xi-P_{n-1})\|\ge \frac{q}{\|Q_n\|}
\ge \|Q_n\|^{-\frac{\nu(\omega_n+1)}{\nu+1}}\geq\|V_n\|^{-\nu}.
$$
When $a=0$,
we obtain
\begin{equation*}
\begin{split}
\|x\xi-y-\theta\|
&=\|(u_n-b)(Q_n\xi-P_n)\|=\|u_n\|\|Q_{n+1}\|^{-1}\\
&> q^{-1}\|Q_n\|^{-\omega_n}\|Q_n\|^{\frac{\omega_n-\nu}{\nu+1}}\\
&\ge q^{-1}\|V_n\|^{-\nu}.
\end{split}
\end{equation*}
We have reached the expected contradiction.

\section{Proof of Theorem \ref{main theorem2}}

We only need to establish the implication ``$\Rightarrow$'' in Theorem  \ref{main theorem2} and
it can be restated as follows.

\begin{thm}
\label{Ostro}
Under the assumption that $\liminf_{k\to\infty}\frac{\log\|Q_k\|}{k}<\infty$,
we have
$$
\dim_H{\Bad}^{\,\varepsilon}(\alpha)<1, \quad \hbox{for any $\varepsilon>0$}.
$$
\end{thm}

\begin{proof}

For positive integers $K, t$, set
$$
{\Bad}_K^{t}(\alpha)=\{\theta \in \mathbb{I}\colon \|Q\|\|\{Q\alpha\}-\theta\|\geq q^{-t}\;
\text{for all}\;Q \;\text{in}\; \mathbb{F}_q[z]\;\text{with}\;\|Q\|\geq \|Q_K\|\}.
$$
For $k \ge 1$, set $n_k = \deg Q_k$.

We define a sequence $(k_i)_{i\geq0}$ as follows.
Set $k_0=K$ and, for $i \geq1$,
let $k_{i+1}$ be the smallest integer $k$ for which $n_k - n_{k_i} > t + 4$.
Since $\|Q_{k+1}\|\geq q \|Q_k\|$,
the sequence $(k_{i+1}-k_i)_{i\ge0}$ is uniformly bounded from above by an absolute constant
and we deduce from our assumption on the growth of the sequence $((\log\|Q_k\| )/ k)_{k \ge 1}$ that
$$
\lambda := \liminf_{i\to \infty}\frac{1}{i}\log\|Q_{k_i}\| < + \infty.
$$

Setting $\Omega(i)=\bigcup\limits_{\substack{\deg Q=n\\ n_{k_i}\le n\le n_{k_{i+1}}-t}}
B(\{Q\alpha\},q^{-n_{k_{i+1}}})$,
we have
$$
\bigcup\limits_{\substack{\deg Q=n\\ n_{k_i}\le n<n_{k_{i+1}}}}B(\{Q\alpha\},q^{-t}\|Q\|^{-1})
=\bigcup\limits_{\substack{\deg Q=n\\ n_{k_i}\le n<n_{k_{i+1}}}}B(\{Q\alpha\},q^{-n-t})
\supset \Omega(i).
$$
Write
$$
\mathcal{C}(k)=\{I(\sigma_1, \cdots,\sigma_k)
\colon (\sigma_1, \cdots,\sigma_k)\in \mathbb{L}_k(\alpha)\},
$$
where $I(\sigma_1, \cdots,\sigma_k)$
is the cylinder of order $n$ with respect to the $\alpha$-expansion
(see at the end of Section 3), and
$$
\mathcal{H}_i=\{B\in \mathcal{C}(k_{i+1})\colon B\cap \Omega(i)=\emptyset\}.
$$
Let
$$
E_i=\bigcup_{B\in \mathcal{H}_i}B\;\;\;\text{and}\;\;\;E=\bigcap_{i\ge1}E_i.
$$
Then we have
$$
{\Bad}_K^t (\alpha)\subset E.
$$

Every ball $B$ in $ \mathcal{C}(k_i)$ can be written as
$B=I(\sigma_1,\cdots,\sigma_{k_i})$
for some $(\sigma_1,\cdots,\sigma_{k_i})$ in $ \mathbb{L}_{k_i}(\alpha)$.
For any $Q$ with $\deg Q=n$ where  $n_{k_i}\le n\le n_{k_{i+1}}-t$,
it follows from Lemma \ref{lemma2} that
\begin{equation}
\label{eq1}
\{Q\alpha\}=\sigma_1D_0+\sigma_2D_1+\cdots+\sigma_{k_i}D_{k_i-1}
+\cdots+\sigma_{k_i+d}D_{k_i+d-1},
\end{equation}
where $d$ is defined by $\|Q_{k_i+d-1}\|\le q^{n_{k_{i+1}}-t}<\|Q_{k_i+d}\|$.
Then, the element of such $\{Q\alpha\}$ contained in the ball $B$
is at least $q^{\deg A_{k_i+1}+\cdots+\deg A_{k_i+d}}$,
which is greater than $q^{n_{k_{i+1}}-n_{k_i}-t}$.
In the same way as one gets equality (\ref{eq1}), we deduce that,
for any distinct $Q,Q'$ in $ \mathbb{F}_q[z]$ with $\deg Q,\deg Q'<n_{k_{i+1}}$,
we have
$$
\|\{Q\alpha\}-\{Q'\alpha\}\|\geq \|D_{k_{i+1}-1}\|=q^{-n_{k_{i+1}}}.
$$
Thus the number of balls $B(\{Q\alpha\},q^{-n_{k_{i+1}}})$
with $ \deg Q=n$ and $n_{k_i}\le n\le n_{k_{i+1}}-t$ which are contained in the ball $B$
is at least $q^{n_{k_{i+1}}-n_{k_i}-t}$.

Then the number of balls in $E_{i+1}$ contained in a ball of $E_i$ is at most
$$
q^{n_{k_{i+1}}-n_{k_i}}-q^{n_{k_{i+1}}-n_{k_i}-t}=(1-q^{-t})q^{n_{k_{i+1}}-n_{k_i}}.
$$
For a real number $s$ in $(0, 1)$, let $H^s$ denote the Hausdorff $s$-measure.
For any $M$ satisfying $\log M>\lambda$, for any $s$ with $1 > s>1+\frac{\log(1-q^{-t})}{\log M}$,
we have
\begin{equation*}
\begin{split}
H^s(E)
&\le\sum_{B\in \cap_{j=1}^iE_j}|B|^s\le (1-q^{-t})^iq^{n_{k_i}}(q^{-n_{k_i}})^s\\
&\le(1-q^{-t})^iM^{(1-s)i}\le1.
\end{split}
\end{equation*}
Then $\dim_H(E)\le 1+\frac{\log(1-q^{-t})}{\log M}<1$,
this completes the proof.

\end{proof}


\section{Proof of Theorem \ref{main theorem3}}

By Theorem \ref{main theorem2},
we only need to prove the following statement.

\begin{thm}
Let $\alpha$ in $\mathbb{F}_q((z^{-1}))$ be an irrational power series and
$(\frac{P_k}{Q_k})_{k\ge1}$ the sequence of its convergents.
Then $\alpha$ is singular on average if and only if $\|Q_k\|^{\frac{1}{k}}$ tends to infinity with $k$.
\end{thm}

\begin{proof}

First, we prove that $\alpha$ is singular on average
under the condition that $\|Q_k\|^{\frac{1}{k}}$ tends to infinity with $k$.

Let $0<c<\frac{1}{q}$ and $k\ge3$ be an integer.
By Lemma \ref{lemma1} and Lemma \ref{lemma2},
for any $Q$ in $\mathbb{F}_q[z]$ with $0<\|Q\|<\|Q_{k+1}\|$,
we have
$Q=B_1Q_0+B_2Q_1+\cdots+B_{k+1}Q_k$.
Then
$$
\{Q\alpha\}=B_1D_0+B_2D_1+\cdots+B_{k+1}D_k,
$$
which gives
$$
\|\{Q\alpha\}\|=\|B_1D_0+B_2D_1+\cdots+B_{k+1}D_k\|
\geq \|D_k\|=\|\{Q_k\alpha\}\|=|\langle Q_k\alpha\rangle|.
$$
In this way,
for each integer $X$ with $\|Q_k\|\leq X<\|Q_{k+1}\|$,
the inequalities
\begin{equation}
\label{neq4}
\|\{h\alpha\}\|\le cX^{-1}\;\;\;\text{and}\;\;\;0<\|h\|\le X
\end{equation}
have a solution in $\mathbb{F}_q[z]$ if and only if $\|\{Q_k\alpha\}\|\le cX^{-1}$.

Thus for each integer $l$ in $[\log_2\|Q_k\|,\log_2\|Q_{k+1}\|)$,
inequalities (\ref{neq4}) have no solution for $X=2^l$ if and only if
$$
-\log_2\frac{\|\{Q_k\alpha\}\|}{c}<l<\log_2\|Q_{k+1}\|.
$$
Since $\|\{Q_k\alpha\}\|=\|Q_{k+1}\|^{-1}$,
the number of integers $l$ in $[\log_2\|Q_k\|,\log_2\|Q_{k+1}\|)$ such that
inequalities (\ref{neq4}) have no solution for $X=2^l$ is at most
$$
\log_2\|Q_{k+1}\|+\log_2\frac{\|\{Q_k\alpha\}\|}{c}+1\le \log\frac{1}{c}+1.
$$
Therefore,
for an integer $N$ with $\log_2\|Q_k\|\le N<\log_2\|Q_{k+1}\|$,
the number of integer $l$ in $\{1,2,\cdots,N\}$ such that
inequalities (\ref{neq4}) have no solution for $X$ is not greater than $(\log\frac{1}{c}+1)(k+1)$.
Recalling that $\Delta_{N,c}(\alpha)$ denote the number of integers $l$ in $\{1,\cdots,N\}$ for
which the inequality $ \|\{Q\alpha\}\|\le c2^{-l} $ has a solution with $0<\|Q\|\le 2^l$,
we have
$$
\frac{N-\Delta_{N,c}(\alpha)}{N}\le\frac{(\log\frac{1}{c}+1)(k+1)}{N}
\le \frac{(\log\frac{1}{c}+1)(k+1)}{\log_2\|Q_k\|}.
$$
By the assumption that $\|Q_k\|^{\frac{1}{k}}$ tends to infinity with $k$,
we deduce that $\frac{N-\Delta_{N,c}(\alpha)}{N}$ converges to $0$.
Therefore,
$\alpha$ is singular on average.

Suppose that $\alpha$ is singular on average,
choose $c=q^{-3}$.
Let $l$ be an integer satisfying $q^{-2}\|Q_{k+1}\|\leq 2^l<\|Q_{k+1}\|$
for some $k\geq 1$.
Then,
we have
$$
\|\{Q_k\alpha\}\|=\|Q_{k+1}\|^{-1}\geq \frac{q^{-2}}{2^l}>\frac{c}{2^l}.
$$
Since $\|\{h\alpha\}\|\geq \|\{Q_k\alpha\}\|$ for any polynomial $h$ with $0<\|h\|<\|Q_{k+1}\|$,
we conclude that inequalities (\ref{neq4}) have no solution for $X=2^l$,
if $l$ is an integer in $[\log_2\|Q_{k+1}\|-2\log_2q,\log_2\|Q_{k+1}\|)$.

By Lemma \ref{lemma1},
$\|Q_{k+1}\|=\Pi_{i=1}^{k+1}\|A_i\|$ and $\deg A_k\ge1$,
we have that $\|Q_{k+1}\|\geq q^2\|Q_{k-1}\|$,
which implies that
$[\log_2\|Q_{k-1}\|-2\log_2q,\log_2\|Q_{k-1}\|)$ and $[\log_2\|Q_{k+1}\|-2\log_2q,\log_2\|Q_{k+1}\|)$
are disjoint for $k\geq 1$.
Let $N$ be an integer with $\log_2\|Q_{2k}\|\leq N<\log_2\|Q_{2k+2}\|$,
it follows that the number of integers $l$ in $\{1,2,\cdots,N\}$ such that
inequalities (\ref{neq4}) have no solution for $X=2^l$ and $c=q^{-3}$ is at least $2k$.
In this way,
$$
\frac{2k}{\log_2\|Q_{2k+2}\|}\leq \frac{2k}{N}\leq\frac{N-\Delta_{N,c}(\alpha)}{N}.
$$
The condition of singularity on average implies that
the right hand side of the above inequality goes to $0$ as $N$ tends to infinity.
By the monotonicity of $(\|Q_k\|)_{k\ge1}$,
we conclude that $(\|Q_k\|^{\frac{1}{k}})_{k\ge1}$ tends to infinity.

\end{proof}

\subsection*{Acknowledgements}
The authors are grateful to the referee for a careful reading.
The second author was supported by NSFC (Grant Nos. 11501168)
and the China Scholarship Council.

IRMA, UMR 7501,
UNIVERSIT\'{E} DE STRASBOURG, CNRS,
7, RUE REN\'{E} DESCARTES,
67000 STRASBOURG, FRANCE;
$bugeaud@math.unistra.fr$  \newline

SCHOOL OF MATHEMATICAL SCIENCES,
HENAN INSTITUTE OF SCIENCE AND TECHNOLOGY
453003 XINXIANG, P. R. CHINA;
$zhliang\_zhang@hotmail.com$

\end{document}